\documentclass[12pt]{amsart}
\usepackage{amsmath,amssymb}
\usepackage{amsfonts}
\usepackage{amsthm}
\usepackage{latexsym}
\usepackage{graphicx}

\def\p{\partial}

\def\R{\mathbb{R}}

\def\vv<#1>{\langle#1\rangle}

\def\XXint#1#2{\setbox0=\hbox{$#1{#2}{\int}$}{#2}\kern-.5\wd0 }

\def\XXint#1#2#3{{\setbox0=\hbox{$#1{#2#3}{\int}$}
     \vcenter{\hbox{$#2#3$}}\kern-.5\wd0}}

\def\vv<#1>{\langle#1\rangle}
\newtheorem{thm}{Theorem}[section]

\newtheorem{prop}{Proposition}[section]
\newtheorem{cor}{Corollary}[section]
\theoremstyle{definition}

\theoremstyle{remark}

\numberwithin{equation}{section}

\begin{document}

\title{Sharp Li-Yau type gradient estimates on hyperbolic spaces}
\author{Chengjie Yu$^1$}
\address{Department of Mathematics, Shantou University, Shantou, Guangdong, 515063, China}
\email{cjyu@stu.edu.cn}
\author{Feifei Zhao}
\address{Department of Mathematics, Shantou University, Shantou, Guangdong, 515063, China}
\email{14ffzhao@stu.edu.cn}
\thanks{$^1$Research partially supported by an NSF project of China with contract no. 11571215.}

\renewcommand{\subjclassname}{%
  \textup{2010} Mathematics Subject Classification}
\subjclass[2010]{Primary 35K05; Secondary 53C44}
\date{}
\keywords{Heat equation, Li-Yau type gradient estimate, heat kernel}
\begin{abstract}
In this paper, motivated by the works of Bakry et. al in finding sharp Li-Yau type gradient estimate for positive solutions of the heat equation on complete Riemannian manifolds with nonzero Ricci curvature lower bound, we first introduce a general form of Li-Yau type gradient estimate and show that the validity of such an estimate for any positive solutions of the heat equation reduces to the validity of the estimate for the heat kernel of the Riemannian manifold.  Then, a sharp Li-Yau type gradient estimate on the three dimensional hyperbolic space is obtained by using the explicit expression of the heat kernel and some optimal Li-Yau type gradient estimates on general hyperbolic spaces are obtained.
\end{abstract}
\maketitle\markboth{Yu \& Zhao}{Li-Yau type gradient estimate}
\section{Introduction}
Let $(M^n,g)$ be a complete Riemannian manifold with Ricci curvature bounded from below by $-k$ with $k$ a nonnegative constant and $u(t,x)$ be a positive solution of the heat equation on $M$. In their seminal work \cite{LY}, Li and Yau showed that
\begin{equation}\label{eq-LY}
\|\nabla \log u\|^2-\alpha\p_t\log u\leq \frac{n\alpha^2}{2t}+\frac{n\alpha^2k}{2(\alpha-1)},
\end{equation}
for any constant $\alpha>1$. When the Ricci curvature is nonnegative, i.e. $k=0$, the Li-Yau estimate \eqref{eq-LY} gives us
\begin{equation}\label{eq-LY-0}
\|\nabla \log u\|^2-\p_t\log u\leq \frac{n}{2t}.
\end{equation}
This estimate is sharp where the equality can be achieved by the fundamental solution of  $\R^n$.

The Li-Yau gradient estimate is of fundamental importance in geometric analysis since a Harnack inequality will be immediately obtained when integrating the estimate on geodesics. For example, by \eqref{eq-LY-0}, one has the following sharp Harnack inequality for complete Riemannian manifolds with nonnegative Ricci curvature:
\begin{equation}
u(t_1,x_1)\leq \left(\frac{t_2}{t_1}\right)^\frac{n}2e^\frac{r^2(x_1,x_2)}{4(t_2-t_1)}u(t_2,x_2),
\end{equation}
for any positive solution $u$ of the heat equation, $0<t_1<t_2$ and $x_1,x_2\in M$.

Comparing to the case of $k=0$ in \eqref{eq-LY}, the Li-Yau gradient estimate \eqref{eq-LY} is not sharp in the case that $k\neq 0$. Finding sharp Li-Yau type gradient estimate for $k\neq 0$ is still an unsolved problem. In a series of works \cite{BQ,BB,BL}, Bakry and his collaborators  provided an interesting approach in finding sharp Li-Yau type gradient estimates. In their works \cite{BB}, Bakry et. al found gradient estimates in a more general form:
\begin{equation}\label{eq-BQ}
\|\nabla \log u\|^2\leq \varphi(t,\p_t\log u)
\end{equation}
with $\varphi(t,s)$ a concave function with respect to $s$. By the first order property of differentiable concave functions:
\begin{equation}
\varphi(t,s)\leq \p_s\varphi(t,s_0)(s-s_0)+\varphi(t,s_0),
\end{equation}
The estimate \eqref{eq-BQ} produces varies of Li-Yau type gradient estimate with time-dependent parameters.

More precisely, in \cite{BB}, Bakry et. al obtained the following estimates for positive solution $u$ of a complete Riemannian manifold $(M^n,g)$ with Ricci curvature lower bound $k$ where $k\neq 0$:
\begin{equation}
\frac{4}{nk}\p_t\log u<1+\frac{\pi^2}{k^2t^2}
\end{equation}
and
\begin{equation}\label{eq-BB}
\|\nabla \log u\|^2<\frac{n}{2}\Phi\left(t,\frac{4}{nk}\p_t\log u\right)
\end{equation}
where
\begin{equation}
\Phi(t,x)=\left\{\begin{array}{ll}\frac{k}{2}(x-2+2\sqrt{1-x}\coth(kt\sqrt{1-x})),&x\leq 1\\
\frac{k}{2}(x-2+2\sqrt{x-1}\cot(kt\sqrt{x-1})),&1\leq x<1+\frac{\pi^2}{k^2t^2}.
\end{array}\right.
\end{equation}
As computed in \cite{BB}, the estimate \eqref{eq-BB} is sharper than Li-Yau's gradient estimate \eqref{eq-LY}, the improvement of \eqref{eq-LY} by Davies in \cite{Da}, Hamilton's estimate in \cite{Ha} and Li-Xu's estimate in \cite{LX}.

In this paper, we further extend a Li-Yau type gradient estimate in the following form:
\begin{equation}\label{eq-BQ-g}
\varphi(t,\|\nabla\log u\|^2)\leq \psi(t,\p_t\log u)
\end{equation}
with $\varphi(t,s)$ increasing and convex with respect to $s$ and $\psi(t,s)$  concave with respect to $s$.

Note that, in \cite{Ya95}, Yau obtained the following gradient estimate:
\begin{equation}\label{eq-Yau}
\|\nabla \log u\|^2-\sqrt{2nk}\sqrt{\|\nabla \log u\|^2+\frac{n}{2t}+2nk}\leq \p_t\log u+\frac{n}{2t}.
\end{equation}
This estimate was later improved in \cite{BQ} by Bakry and Qian to the following estimate:
\begin{equation}\label{eq-BQ-nl}
\|\nabla \log u\|^2-\sqrt{nk}\sqrt{\|\nabla \log u\|^2+\frac{n}{2t}+\frac{nk}{4}}\leq \p_t\log u+\frac{n}{2t}.
\end{equation}
It is clear that the left hand sides of \eqref{eq-Yau} and \eqref{eq-BQ-nl} are both increasing and convex functions of $\|\nabla \log u\|^2$. So, the gradient estimates \eqref{eq-Yau} and \eqref{eq-BQ-nl} are both special cases of the general form \eqref{eq-BQ-g}.

Then, motivated by our previous work \cite{YZ}, we found that to show the validity of \eqref{eq-BQ-g} for any positive solutions of the heat equation, one only need to show the validity of \eqref{eq-BQ-g} for the heat kernel. More precisely, we have the following conclusion.
\begin{thm}\label{thm-general}
Let $(M^n,g)$ be a complete Riemannian manifold with Ricci curvature bounded from below and
\begin{equation}
\mathcal P(M,g)=\{u\in C^\infty(\R^+\times M)\ |\ u>0\ \mbox{and}\ \p_tu-\Delta_g u=0\}.
\end{equation}
Moreover, let $\varphi(t,s)\in C(\R^+\times [0,\infty))$ be increasing and convex with respect to $s$ and $\psi(t,s)\in C(\R^+\times \R) $ be concave with respect to $s$. Then,
\begin{equation}
\varphi(t,\|\nabla\log u\|^2)\leq \psi(t,\p_t\log u)
\end{equation}
for any $u\in \mathcal P(M,g)$ if and only if
\begin{equation}
\varphi(t,\|\nabla_x\log H(t,x,y)\|^2)\leq\psi(t,\p_t\log H(t,x,y))
\end{equation}
for any $t>0$ and $x,y\in M$, where $H(t,x,y)$ is the heat kernel of $(M,g)$.
\end{thm}
By Theorem \ref{thm-general} and the expression of heat kernel of the three dimensional hyperbolic space $\mathbb H^3$ (see \cite{DM,Da}):
\begin{equation}\label{eq-K3}
K_3(t,r)=(4\pi t)^{-\frac 32}e^{-\frac{r^2}{4t}-t}\frac{r}{\sinh r},
\end{equation}
we are able to derive a sharp Li-Yau type gradient estimate for three dimensional hyperbolic manifolds.
\begin{thm}\label{thm-dim-3}
Let $(M^3,g)$ be a complete Riemannian manifold with constant sectional curvature $-1$. Then, for any $u\in \mathcal P(M,g)$,
\begin{equation}\label{eq-dim-3-1}
\p_t\log u+\frac{3}{2t}+1\geq 0
\end{equation}
and
\begin{equation}\label{eq-dim-3-2}
\begin{split}
\|\nabla \log u\|\leq &\sqrt{\p_t\log u+\frac{3}{2t}+1}+Z\left(2t\sqrt{\p_t\log u+\frac{3}{2t}+1}\right)
\end{split}
\end{equation}
where $Z(r)=\coth r-\frac1r$. The equality of the estimate \eqref{eq-dim-3-2} holds when $M$ is the hyperbolic space $\mathbb{H}^3$ and $u$ is the heat kernel. In particular, by noting that $0\leq Z\leq 1$, we have
\begin{equation}\label{eq-dim-3-2-s}
\begin{split}
\|\nabla \log u\|\leq &\sqrt{\p_t\log u+\frac{3}{2t}+1}+1.
\end{split}
\end{equation}
\end{thm}
Furthermore, by using the first order property of concave functions, one has the following linearization of \eqref{eq-dim-3-2}:
\begin{equation}
\begin{split}
\|\nabla \log u\|^2\leq& \left(1+2\left(Z'(r_0)+\frac{Z(r_0)}{r_0}\right)t+\frac{4Z(r_0)Z'(r_0)}{r_0}t^2\right)\left(\p_t\log u+\frac{3}{2t}+1-\frac{r_0^2}{4t^2}\right)\\
&+\left(\frac{r_0}{2t}+Z(r_0)\right)^2
\end{split}
\end{equation}
for any $r_0\geq 0$. In particular, when $r_0=0$, one has
\begin{equation}
\|\nabla \log u\|^2\leq \left(1+\frac{2}{3}t\right)^2\left(\p_t\log u+\frac{3}{2t}+1\right).
\end{equation}

We believe that there is a similar result as in Theorem \ref{thm-dim-3} for arbitrary dimensional hyperbolic manifolds. However, because the expressions of the heat kernels are much more involved, one can not write them in such an explicit form as in \eqref{eq-dim-3-2}. We will discuss this in another place. Instead, by using the estimates in \cite{YZ}, we are able to extend \eqref{eq-dim-3-2-s} to general hyperbolic manifolds.
\begin{thm}\label{thm-general-h}
Let $(M^n,g)$ be a complete Riemannian manifold with constant sectional curvature $-1$. Then,
\begin{enumerate}
\item when $n$ is odd,
\begin{equation}\label{eq-odd-1}
\p_t\log u+\frac{n}{2t}+\frac{(n-1)^2}{4}\geq 0
\end{equation}
and
\begin{equation}\label{eq-odd-2}
\|\nabla\log u\|\leq \sqrt{\p_t\log u+\frac{n}{2t}+\frac{(n-1)^2}{4}}+\frac{n-1}{2}
\end{equation}
for any $u\in \mathcal P(M,g)$;
\item when $n$ is even,
\begin{equation}\label{eq-even-1}
\p_t\log u+\frac{n+1}{2t}+\frac{(n-1)^2}{4}\geq 0
\end{equation}
and
\begin{equation}\label{eq-even-2}
\|\nabla\log u\|\leq \sqrt{\p_t\log u+\frac{n+1}{2t}+\frac{(n-1)^2}{4}}+\frac{n-1}{2}
\end{equation}
for any $u\in \mathcal P(M,g)$;
\end{enumerate}
\end{thm}

Note that \eqref{eq-odd-1} and \eqref{eq-even-1} were shown in \cite{YZ}, and as mentioned in \cite{YZ}, the estimate \eqref{eq-odd-1} is not true for the hyperbolic plane. Moreover, note that  \eqref{eq-odd-2} and \eqref{eq-even-2} are sharper the optimal Li-Yau type gradient estimates in \cite{YZ}. Furthermore, by using a classical argument similarly as in \cite{LY} and the estimates \eqref{eq-odd-2} and \eqref{eq-even-2}, one can obtain the same Harnack inequalities as in \cite{YZ}:
\begin{equation}
u(t_1,x_1)\leq \left(\frac{t_2}{t_1}\right)^\frac n2 \exp\left({\frac{r^2(x_1,x_2)}{4(t_2-t_1)}+\frac{(n-1)^2}{4}(t_2-t_1)+\frac{n-1}{2}r(x_1,x_2)}\right)u(t_2,x_2)
\end{equation}
for any $0<t_1<t_2$ and $x_1,x_2\in M$ with $(M^n,g)$ an odd dimensional complete Riemannian manifold with constant sectional curvature $-1$,  and
\begin{equation}
u(t_1,x_1)\leq \left(\frac{t_2}{t_1}\right)^\frac {n+1}2 \exp\left({\frac{r^2(x_1,x_2)}{4(t_2-t_1)}+\frac{(n-1)^2}{4}(t_2-t_1)+\frac{n-1}{2}r(x_1,x_2)}\right)u(t_2,x_2)
\end{equation}
for any $0<t_1<t_2$ and $x_1,x_2\in M$ with $(M^n,g)$ an even dimensional complete Riemannian manifold with constant sectional curvature $-1$. Here $r(x_1,x_2)$ means the distance between $x_1$ and $x_2$ and $u\in \mathcal P(M,g)$.

There is a rich literature in extending the Li-Yau gradient estimate in different settings or in different forms. See for examples \cite{Ca,CN,CH,NN,Ha2,Le,Pe,Qi,Ya94,ZZ0,ZZ,ZZh}.

The rest of this pare is organized as follows: in Section 2, we prove Theorem \ref{thm-general}, in Section 3, we prove Theorem \ref{thm-dim-3} and its linearization, and finally in Section 4, we prove Theorem \ref{thm-general-h} and some of its corollaries.

\section{Li-Yau type gradient estimate in a general form}
In this section, we prove Theorem \ref{thm-general}. First of all, we have the following finite sum version of Theorem \ref{thm-general}.
\begin{prop}
Let $(M^n,g)$ be a complete Riemannian manifold and $u,v\in \mathcal P(M,g)$. Moreover, let $\varphi(t,s)\in C(\R^+\times [0,+\infty))$ be increasing and convex with respect to $s$ and $\psi(t,s)\in C(\R^+\times \R)$ be concave with respect to $s$. Suppose that
\begin{equation}
\varphi(t,\|\nabla \log u\|^2)\leq \psi(t,\p_t\log u)
\end{equation}
and
\begin{equation}
\varphi(t,\|\nabla \log v\|^2)\leq \psi(t,\p_t\log v).
\end{equation}
Then,
\begin{equation}
\varphi(t,\|\nabla \log (u+v)\|^2)\leq \psi(t,\p_t\log(u+v)).
\end{equation}
\end{prop}
\begin{proof}By the assumptions,
\begin{equation}
\begin{split}
&\varphi(t,\|\nabla \log (u+v)\|^2)\\
=&\varphi\left(t,\frac{\|\nabla \log u\|^2u^2+\|\nabla\log v\|^2v^2+2\vv<\nabla\log u,\nabla \log v>uv}{(u+v)^2}\right)\\
\leq&\varphi\left(t,\frac{\|\nabla \log u\|^2u^2+\|\nabla\log v\|^2v^2+(\|\nabla\log u\|^2+\|\nabla \log v\|^2)uv}{(u+v)^2}\right)\\
\leq&\frac{u}{u+v}\varphi(t,\|\nabla\log u\|^2)+\frac{v}{u+v}\varphi(t,\|\nabla\log v\|^2)\\
\leq&\frac{u}{u+v}\psi(t,\p_t\log u)+\frac{v}{u+v}\psi(t,\p_t\log v)\\
\leq&\psi\left(t,\frac{u\p_t\log u+v\p_t\log v}{u+v}\right)\\
=&\psi(t,\p_t\log (u+v)).
\end{split}
\end{equation}
\end{proof}

Next, we come to prove Theorem \ref{thm-general}.
\begin{proof}[Proof of Theorem \ref{thm-general}.]
The only if part is clear. For the if part, similarly as in \cite{YZ}, we only need to verify that
\begin{equation}
\varphi(t,\|\nabla \log u\|^2)\leq \psi(t,\p_t\log u)
\end{equation}
for any
\begin{equation}
u(t,x)=\int_M H(t,x,y)f(y)dy
\end{equation}
with $f$ an arbitrary  nonnegative smooth function with compact support. Then,
\begin{equation}
\begin{split}
&\varphi(t,\|\nabla \log u\|^2)\\
=&\varphi\left(t,\left(\int_M H(t,x,y) f(y)dy\right)^{-2}\left\|\int_M \nabla_xH(t,x,y) f(y)dy\right\|^2\right)\\
=&\varphi\Bigg(t,\left(\int_M\int_M H(t,x,y)H(t,x,z)f(y)f(z)dydz\right)^{-1}\times\\
&\int_M \int_M\vv<\nabla_x\log H(t,x,y),\nabla_x\log H(t,x,z)> H(t,x,y)H(t,x,z)f(y)f(z) dydz\Bigg)\\
\leq&\varphi\Bigg(t,\left(\int_M\int_M H(t,x,y)H(t,x,z)f(y)f(z)dydz\right)^{-1}\times\\
&\frac12\int_M \int_M\left(\|\nabla_x\log H(t,x,y)\|^2+\|\nabla_x\log H(t,x,z)\|^2\right) H(t,x,y)H(t,x,z)f(y)f(z) dydz\Bigg)\\
=&\frac12\left(\int_M\int_M H(t,x,y)H(t,x,z)f(y)f(z)dydz\right)^{-1}\times\int_M \int_M\big(\varphi(t,\|\nabla_x\log H(t,x,y)\|^2)+\\
&\varphi(t,\|\nabla_x\log H(t,x,z)\|^2)\big) H(t,x,y)H(t,x,z)f(y)f(z) dydz\\
\leq&\frac12\left(\int_M\int_M H(t,x,y)H(t,x,z)f(y)f(z)dydz\right)^{-1}\times\int_M \int_M\big(\psi(t,\p_t\log H(t,x,y))+\\
&\psi(t,\p_t\log H(t,x,z)\big) H(t,x,y)H(t,x,z)f(y)f(z) dydz\\
\leq&\frac{1}{2}\psi\left(t,\frac{\int_M\int_M\p_t\log H(t,x,y)H(t,x,y)H(t,x,z)f(y)f(z) dydz}{\int_M\int_M H(t,x,y)H(t,x,z)f(y)f(z)dydz}\right)+\\
&\frac{1}{2}\psi\left(t,\frac{\int_M\int_M\p_t\log H(t,x,z)H(t,x,y)H(t,x,z)f(y)f(z) dydz}{\int_M\int_M H(t,x,y)H(t,x,z)f(y)f(z)dydz}\right)\\
=&\psi(t,\p_t\log u).
\end{split}
\end{equation}
This completes the proof of the theorem.
\end{proof}
\section{A sharp Li-Yau type gradient estimate on $\mathbb{H}^3$}
In this section, with the help of Theorem \ref{thm-general}, we prove Theorem \ref{thm-dim-3} and its linearization.
\begin{proof}[Proof of Theorem \ref{thm-dim-3}]
By taking universal cover, we only need to prove the theorem for $\mathbb H^3$. By \eqref{eq-K3},  we have
\begin{equation}\label{eq-Y}
Y:=\|\nabla\log K_3\|^2=\left(\frac{r}{2t}+Z(r) \right)^2
\end{equation}
and
\begin{equation}\label{eq-X}
X:=\p_t\log K_3=-\frac{3}{2t}+\frac{r^2}{4t^2}-1.
\end{equation}
By the last equation, we have
\begin{equation}
\p_t\log K_3\geq-\frac{3}{2t}-1.
\end{equation}
Then, by applying theorem \ref{thm-general} with $\varphi(t,s)=-\frac{3}{2t}-1$ and $\psi(t,s)=s$, we get \eqref{eq-dim-3-1}.

By substituting \eqref{eq-X} into \eqref{eq-Y}, we know that $Y$ is a function of $X$. To prove \eqref{eq-dim-3-2}, by Theorem \ref{thm-general}, we only need to show the function is concave. Note that
\begin{equation}\label{eq-YX}
\begin{split}
\frac{d^2Y}{dX^2}=&16t^4\frac{d^2Y}{(dr^2)^2}\\
=&\frac{4t^3}{r^3}
\left({r^2Z''(r)}+{rZ'(r)}-{Z(r)}\right)+\frac{8t^4}{r^3}\left({rZ(r)Z''(r)+rZ'(r)^2-Z(r)Z'(r)}\right),\\
\end{split}
\end{equation}
and
\begin{equation}\label{eq-YX-1}
\begin{split}
&{r^2Z''(r)}+{rZ'(r)}-{Z(r)}\\
=&\frac{1}{\sinh^3r}(2r^2\cosh r-r\sinh r-\cosh r\sinh^2 r)\\
=&\frac{1}{\sinh^3r}\left(\sum_{k=0}^{\infty}\frac{2}{(2k)!}r^{2k+2}-\sum_{k=1}^\infty\frac{1}{(2k-1)!}r^{2k}-\frac{1}{4}(\cosh(3r)-\cosh(r))\right)\\
=&\frac{1}{\sinh^3r}\left(\sum_{k=0}^{\infty}\frac{2}{(2k)!}r^{2k+2}-\sum_{k=1}^\infty\frac{1}{(2k-1)!}r^{2k}-\frac{1}{4}\sum_{k=0}^\infty\frac{(9^k-1)}{(2k)!}r^{2k}\right)\\
=&\frac{1}{\sinh^3r}\sum_{k=3}^\infty\left(\frac{2}{(2k-2)!}-\frac{1}{(2k-1)!}-\frac{(9^k-1)}{4(2k)!}\right)r^{2k}\\
=&\frac{1}{\sinh^3r}\sum_{k=3}^\infty\frac{(32k^2-24k+1)-9^k}{4(2k)!}r^{2k}\\
\leq&0
\end{split}
\end{equation}
since
\begin{equation}
(32k^2-24k+1)-9^k\leq (32k^2-24k+1)-81k^2<0
\end{equation}
when $k\geq 3$.

Furthermore,
\begin{equation}\label{eq-YX-2}
\begin{split}
&rZ(r)Z''(r)+rZ'(r)^2-Z(r)Z'(r)\\
=&\frac{1}{r^3\sinh^4r}(4\sinh^4r+2r^4\cosh^2r+r^4-3r\cosh r\sinh^3r-3r^2\sinh^2r-r^3\sinh r\cosh r)
\end{split}
\end{equation}
and
\begin{equation}
\begin{split}
&4\sinh^4r+2r^4\cosh^2r+r^4-3r\cosh r\sinh^3r-3r^2\sinh^2r-r^3\sinh r\cosh r\\
=&\frac12\cosh(4r)-2\cosh(2r)+\frac32+r^4(\cosh(2r)+1)+r^4-3r\left(\frac18\sinh(4r)-\frac14\sinh(2r)\right)\\
&-\frac{3r^2}{2}(\cosh(2r)-1)-\frac{r^3}{2}\sinh(2r)\\
=&\sum_{k=5}^\infty\frac{2^{2k-3}(-(3k-8)2^{2k-1}+8k^4-28k^3+16k^2+4k-16)}{(2k)!}r^{2k}\\
\leq&0
\end{split}
\end{equation}
since
\begin{equation}
\begin{split}
-(3k-8)2^{2k-1}+8k^4-28k^3+16k^2+4k-16\leq-10\times 2^{2k-1}+8k^4<0
\end{split}
\end{equation}
when $k\geq 6$ and
 \begin{equation}
\begin{split}
-(3k-8)2^{2k-1}+8k^4-28k^3+16k^2+4k-16|_{k=5}=-1680<0.
\end{split}
\end{equation}
 We complete the proof of the theorem by substituting \eqref{eq-YX-1} and \eqref{eq-YX-2} in to \eqref{eq-YX}. The fact that $0\leq Z\leq 1$ comes from \cite[Proposition 3.2]{YZ}.
\end{proof}

By using the first order property of differentiable concave functions, we have the following linearization of \eqref{eq-dim-3-2}.
\begin{thm}
Let $(M^3,g)$ be a complete Riemannian manifold with constant sectional curvature $-1$, and $u\in \mathcal P(M,g)$. Then,
\begin{equation}\label{eq-linear}
\begin{split}
\|\nabla \log u\|^2\leq& \left(1+2\left(Z'(r_0)+\frac{Z(r_0)}{r_0}\right)t+\frac{4Z(r_0)Z'(r_0)}{r_0}t^2\right)\left(\p_t\log u+\frac{3}{2t}+1-\frac{r_0^2}{4t^2}\right)\\
&+\left(\frac{r_0}{2t}+Z(r_0)\right)^2
\end{split}
\end{equation}
for any $r_0\geq 0$, where $Z(r)=\coth r-\frac 1r$. In particular, when $r_0=0$, we have
\begin{equation}\label{eq-linear-0}
\|\nabla \log u\|^2\leq \left(1+\frac{2}{3}t\right)^2\left(\p_t\log u+\frac{3}{2t}+1\right).
\end{equation}
\end{thm}
\begin{proof} Let the notations be the same as in the proof of Theorem \ref{thm-dim-3}. Then
\begin{equation}
\begin{split}
\frac{dY}{dX}=&2t^2\frac{dY}{rdr}=1+2\left(Z'(r)+\frac{Z(r)}{r}\right)t+\frac{4Z(r)Z'(r)}{r}t^2.
\end{split}
\end{equation}
Since $Y$ is a concave function of $X$, we have
\begin{equation}
\begin{split}
&Y(t,X)\\
\leq&\frac{dY}{dX}(X-X_0)+Y(t, X_0)\\
=& \left(1+2\left(Z'(r_0)+\frac{Z(r_0)}{r_0}\right)t+\frac{4Z(r_0)Z'(r_0)}{r_0}t^2\right)\left(X+\frac{3}{2t}+1-\frac{r_0^2}{4t^2}\right)+\left(\frac{r_0}{2t}+Z(r_0)\right)^2.
\end{split}
\end{equation}
where $X_0=-\frac{3}{2t}+\frac{r_0^2}{4t^2}-1$. Combining this and \eqref{eq-dim-3-2}, we get \eqref{eq-linear}.

When $r_0=0$, noting that $Z(0)=0$ and $Z'(0)=\frac13$, we get \eqref{eq-linear-0}.
\end{proof}
\section{Optimal Li-Yau type gradient estimates on general hyperbolic spaces}
Before proving Theorem \ref{thm-general-h}, recall some facts of the heat kernels of hyperbolic spaces. The same as in \cite{DM}, write the heat kernel $K_n(t,r(x,y))$ of $\mathbb H^n$ in the form:
\begin{equation}\label{eq-Kn}
K_n(t,r)=(4\pi t)^{-\frac{n}2}e^{-\frac{r^2}{4t}-\frac{(n-1)^2}{4}t}\alpha_n(t,r).
\end{equation}
Then, as shown in \cite{DM,YZ}, we have
\begin{equation}\label{eq-odd-pt}
\p_t\alpha_n\geq 0
\end{equation}
when $n$ is odd,
\begin{equation}\label{eq-even-pt}
\p_t(t^\frac12\alpha_n)\geq 0
\end{equation}
when $n$ is even. Moreover,
\begin{equation}\label{eq-pr}
0\leq -\p_r\log\alpha_n\leq \frac{n-1}{2}.
\end{equation}
\begin{proof}[Proof of Theorem \ref{thm-general-h}] By taking universal cover, we only need to prove the theorem for the hyperbolic space $\mathbb H^n$. Moreover, we only need to prove the odd dimensional case. The proof of the even dimensional case is similar by using \eqref{eq-even-pt} and \eqref{eq-pr}.

By \eqref{eq-Kn} and \eqref{eq-pr},
\begin{equation}\label{eq-nabla}
\|\nabla\log K_n\|^2=\left(\frac{r}{2t}-\p_r\log \alpha_n\right)^2\leq \left(\frac{r}{2t}+\frac{n-1}{2}\right)^2.
\end{equation}
Moreover, by \eqref{eq-Kn} and \eqref{eq-odd-pt},
\begin{equation}\label{eq-pt}
\p_t\log K_n=-\frac{n}{2t}-\frac{(n-1)^2}{4}+\frac{r^2}{4t^2}+\p_t\log \alpha_n\geq-\frac{n}{2t}-\frac{(n-1)^2}{4}+\frac{r^2}{4t^2}.
\end{equation}
Thus,
\begin{equation}
\p_t\log K_n\geq -\frac{n}{2t}-\frac{(n-1)^2}{4}.
\end{equation}
By applying Theorem \ref{thm-general} with $\varphi(t,s)=-\frac{n}{2t}-\frac{(n-1)^2}{4}$ and $\psi(t,s)=s$, we get \eqref{eq-odd-1}. Moreover, by \eqref{eq-nabla} and \eqref{eq-pt},
\begin{equation}
\|\nabla \log K_n\|^2\leq \left(\sqrt{\p_t\log K_n +\frac{n}{2t}+\frac{(n-1)^2}{4}}+\frac{n-1}{2}\right)^2.
\end{equation}
Noting that $\psi(t,s)=\left(\sqrt{s+\frac{n}{2t}+\frac{(n-1)^2}{4}}+\frac{n-1}{2}\right)^2$ is a concave function of $s$, we get \eqref{eq-odd-2} by Theorem \ref{thm-general}.
\end{proof}
From Theorem \ref{thm-general-h}, we can obtain the optimal Li-Yau type gradient estimate in \cite{YZ} for hyperbolic manifolds.
\begin{cor}
Let $(M^n,g)$ be a complete Riemannian manifold with constant sectional curvature $-1$. Then,
\begin{enumerate}
\item when $n$ is odd,
\begin{equation}
\beta\|\nabla \log u\|^2-\p_t\log u\leq\frac{n}{2t}+\frac{(n-1)^2}{4(1-\beta)}
\end{equation}
for any $\beta\in [0,1)$ and $u\in \mathcal P(M,g)$;
\item when $n$ is even,
\begin{equation}
\beta\|\nabla \log u\|^2-\p_t\log u\leq\frac{n+1}{2t}+\frac{(n-1)^2}{4(1-\beta)}
\end{equation}
for any $\beta\in [0,1)$ and $u\in \mathcal P(M,g)$;
\end{enumerate}
\end{cor}
\begin{proof}
We only need to prove the odd dimensional case. The proof of the even dimensional case is similar.

By \eqref{eq-odd-2}, we have
\begin{equation}
\begin{split}
&\beta\|\nabla \log u\|^2-\p_t\log u\\
\leq&\beta\left(\sqrt{\p_t\log u +\frac{n}{2t}+\frac{(n-1)^2}{4}}+\frac{n-1}{2}\right)^2-\p_t\log u\\
=&\beta\left(X+\frac{n-1}{2}\right)^2-X^2+\frac{n}{2t}+\frac{(n-1)^2}{4}\\
=&-(1-\beta)\left(X-\frac{(n-1)\beta}{2(1-\beta)}\right)^2+\frac{n}{2t}+\frac{(n-1)^2}{4(1-\beta)}\\
\leq&\frac{n}{2t}+\frac{(n-1)^2}{4(1-\beta)}.\\
\end{split}
\end{equation}Here $X=\sqrt{\p_t\log u +\frac{n}{2t}+\frac{(n-1)^2}{4}}$. This completes the proof of the odd dimensional case.
\end{proof}
Moreover, we can obtain the Harnack inequalities in \cite{YZ} by using Theorem \ref{thm-general-h} directly via a classical argument in \cite{LY}.
\begin{cor}
Let $(M^n,g)$ be a complete Riemannian manifold with constant sectional curvature $-1$. Then,
\begin{enumerate}
\item when $n$ is odd,
\begin{equation}
u(t_1,x_1)\leq \left(\frac{t_2}{t_1}\right)^\frac n2 \exp\left({\frac{r^2(x_1,x_2)}{4(t_2-t_1)}+\frac{(n-1)^2}{4}(t_2-t_1)+\frac{n-1}{2}r(x_1,x_2)}\right)u(t_2,x_2)
\end{equation}
for any $0<t_1<t_2$ and $x_1,x_2\in M$;
\item when $n$ is even,
\begin{equation}
u(t_1,x_1)\leq \left(\frac{t_2}{t_1}\right)^\frac {n+1}2 \exp\left({\frac{r^2(x_1,x_2)}{4(t_2-t_1)}+\frac{(n-1)^2}{4}(t_2-t_1)+\frac{n-1}{2}r(x_1,x_2)}\right)u(t_2,x_2)
\end{equation}
for any $0<t_1<t_2$ and $x_1,x_2\in M$.
\end{enumerate}
Here $r(x_1,x_2)$ means the distance between $x_1$ and $x_2$, and $u\in \mathcal P(M,g)$.
\end{cor}
\begin{proof}
We only need to prove the odd dimensional case. The proof of the even dimesional case is similar.

By \eqref{eq-odd-2},
\begin{equation}
\begin{split}
&\log u(t_2,x_2)-\log u(t_1,x_1)\\
=&\int_{t_1}^{t_2}d\log u(t,\gamma(t))\\
=&\int_{t_1}^{t_2}(\p_t\log u(t,\gamma(t))+\vv<\nabla\log u(t,\gamma(t)),\gamma'(t)>)dt\\
\geq&\int_{t_1}^{t_2}(\p_t\log u(t,\gamma(t))-\|\nabla\log u(t,\gamma(t))\|\|\gamma'(t)\|)dt\\
\geq&\int_{t_1}^{t_2}\left(\p_t\log u(t,\gamma(t))-\left(\sqrt{\p_t\log u(t,\gamma(t)) +\frac{n}{2t}+\frac{(n-1)^2}{4}}+\frac{n-1}{2}\right)\|\gamma'(t)\|\right)dt\\
=&\int_{t_1}^{t_2}\left(-\frac{n}{2t}-\frac{(n-1)^2}{4}+X^2-\frac{r(x_1,x_2)}{t_2-t_1}\left(X+\frac{n-1}{2}\right)\right)dt\\
\geq&-\frac{n}{2}\log\left(\frac{t_2}{t_1}\right)-\frac{(n-1)^2}{4}(t_2-t_1)-\frac{n-1}{2}r(x_1,x_2)-\frac{r^2(x_1,x_2)}{4(t_2-t_1)}.
\end{split}
\end{equation}
Here $\gamma$ is a minimal geodesic  of $(M,g)$ with $\gamma(t_1)=x_1$ and $\gamma(t_2)=x_2$, and $X=\sqrt{\p_t\log u(t,\gamma(t)) +\frac{n}{2t}+\frac{(n-1)^2}{4}}$. This completes the proof the odd dimensional case.
\end{proof}

\end{document}